\documentclass[12pt,reqno,a4paper]{article}

\usepackage{amsmath,amsthm,amssymb}
\usepackage{color}

\usepackage{enumitem}

\usepackage[T1]{fontenc}
\usepackage{graphics}
\usepackage{fullpage}
\usepackage[colorlinks,linkcolor=blue]{hyperref}
\usepackage{mathtools}
\usepackage{multicol}
\usepackage[colorinlistoftodos,prependcaption]{todonotes}
\usepackage[all]{xy}

\makeatletter
\DeclareRobustCommand{\em}{%
  \@nomath\em \if b\expandafter\@car\f@series\@nil
  \normalfont \else \bfseries\itshape \fi}
\makeatother

\newtheorem{theorem}[subsection]{Theorem}%[section]

\newtheorem{corollary}[subsection]{Corollary}
\newtheorem{lemma}[subsection]{Lemma}
\newtheorem{proposition}[subsection]{Proposition}
\theoremstyle{definition}

\newtheorem{example}[subsection]{Example}
\theoremstyle{remark}
\newtheorem{remark}[subsection]{Remark}

\numberwithin{equation}{section}
\numberwithin{figure}{section}

\newcommand{\B}[1]{{\mathbf #1}}
\newcommand{\C}[1]{{\mathcal #1}}

\newcommand{\OP}{\operatorname}

\begin{document}

\title{On Lipschitz functions on groups equipped with conjugation-invariant norms}
\author{Jarek K\k{e}dra}
%\address{University of Aberdeen and University of Szczecin}
%\email{kedra@abdn.ac.uk}

%\thanks{thanks} 
%\keywords{keywords}
%\subjclass[2000]{Primary 20F65, 51F30; Secondary 57}

\maketitle

\begin{abstract}
We observe that a function on a group equipped with a bi-invariant word metric
is Lipschitz if and only if it is a partial quasimorphism bounded on the generating set.
We also show that an undistorted element is always detected by an antisymmetric homogeneous 
partial quasimorphism. We provide a general homogenisation procedure for Lipschitz
functions and relate partial quasimorphisms on a group to ones on its asymptotic cones.
\end{abstract}

\section{Introduction}

\subsubsection*{Lipschitz functions vs partial quasimorphisms}
Let $(G,d)$ be a metric group, where $d$ is bi-invariant metric. 
Let $\|g\|=d(g,1)$ denote the corresponding {\em norm}.
A function $f\colon G\to \B R$ is called:
\begin{itemize}
\item
a {\em partial quasimorphism} relative to the norm $\|\ \|$
if there exists a constant $D\geq 0$ such that
\begin{equation}
|\delta f(g,h)| = |f(g)-f(gh)+f(h)| \leq D\min\{\|g\|,\|h\|\},
\label{Eq:def-pqm}
\end{equation}
for all $g,h\in G$;
\item 
{\em homogeneous} if
$f(g^n) = nf(g)$ for all $g\in G$ and $n\in \B N$;
\item
{\em antisymmetric} if $f(g^{-1}) = -f(g)$ for all $g\in G$.
\end{itemize}
An element $g\in G$ is called
{\em undistorted} with respect to the norm $\|\ \|$ 
if $\|g^n\|\geq Cn$ for some $C>0$ and all $n\in \B N$.

\begin{theorem}\label{T:main}
Let a group $G$ be equipped with a bi-invariant word metric $d$.  A function
$f\colon G\to \B R$ is Lipschitz if and only if it is a partial quasimorphism
relative to $\|\ \|$ and it is bounded on the generating set.  Moreover, an
element $g\in G$ is undistorted if and only if there exists an antisymmetric 
homogeneous partial quasimorphism $f\colon G\to \B R$ such that $f(g)>0$
and $f$ is bounded on the generating set.
\end{theorem}

\begin{remark}
A word metric with respect to a symmetric generating set $S\subseteq G$ is
bi-invariant if and only if $S$ is normal, that is,  $S=g^{-1}Sg$ for
all $g\in G$.  Such generating sets are infinite unless $G$ is virtually
abelian and finitely generated.
\end{remark}

\begin{remark}
A function $f\colon G\to \B R$ is antisymmetric and homogeneous if and only if 
$f(g^n)=nf(g)$ for all $n\in \B Z$ and $g\in G$.
Since quasimorphisms (see Example \ref{E:qm} below) are nearly antisymmetric,
their homogeneity is usually defined with integer exponents \cite[Section 2.2.2]{MR2527432}.
This is not always the case for partial quasimorphims. For example, the translation
length (see Example \ref{E:norm-qm})
is a homogeneous partial quasimorphism. Since it is non-negative, it is not antisymmetric
unless it is zero. See also \cite[Theorem 1.3]{MR2968955} for an example in
symplectic geometry.
\end{remark}

\begin{remark}
It is well known that partial quasimorphisms bounded on generatings sets are
Lipschitz. The above theorem shows the opposite.  Homogeneous
partial quasimorphsms are a common tool to detect undistorted elements and what
we observe above is that {\it every} undistorted element is detected by an {\it antisymmetric}
homogeneous partial quasimorphism.
\end{remark}

\begin{remark}
Considering word norms is not very restrictive. We show in Lemma \ref{L:length-word}
that bi-invariant length metric is equivalent to a suitable word metric. Moreover,
the most natural bi-invariant metrics are either word metrics (e.g., commutator
length, verbal length, fragmentation norm, autonomous norm etc) or length metrics
(e.g. Hofer's norm).
\end{remark}

\begin{example}\label{E:qm}
A function $f\colon G\to \B R$ is called a {\em quasimorphism} if there
exists $D\geq 0$ such that
$$
|f(g)-f(gh)+f(h)|\leq D
$$
for all $g,h\in G$. It is, obviously, a partial quasimorphism.
In \cite{MR3426433}, there are examples of groups $G$ equipped with bi-invariant
word metrics such that
an element $g\in G$ generates an unbounded cyclic subgroup if and only if
it is detected by a homogeneous quasimorphism. 
\hfill $\diamondsuit$
\end{example}

\begin{example}\label{E:Binfty}
The commutator subgroup $[\B B_{\infty},\B B_{\infty}]$ of the infinite braid
group is perfect and does not admit unbounded quasimorphisms \cite{MR2509718}.
On the other hand, it contains elements undistorted with respect to a bi-invariant
word metric \cite{MR3426696}. It follows from Theorem~\ref{T:main} that such
elements are detected by homogeneous partial quasimorphisms. A concrete example
was constructed by Kimura \cite{MR3809596}.
\hfill $\diamondsuit$
\end{example}

A common construction of quasimorphisms is due to Brooks and it is well known
that many groups admit an abundance of quasimorphisms \cite{MR2527432}. For
example, the space of quasimorphisms on a non-elementary hyperbolic group is
infinite dimensional.  On the other hand, systematic constructions of partial
quasimorphisms are rare.  One source is provided by Floer theoretic spectral
invariants in symplectic geometry \cite{MR2208798,MR2968955} and another uses
quasimorphisms \cite{MR4470192,MR3523258,MR3809596}. It would be useful to
know more constructions for finitely generated groups.

\subsubsection*{Homogeneous partial quasimorphisms}

A {\em homogenisation} $\widehat{f}$ of a function $f\colon G\to \B R$ is defined
by
\begin{equation}
\widehat{f}(g) = \lim_{n\to \infty}\frac{f(g^n)}{n}
\label{Eq:homo}
\end{equation}
provided that the limit exists.

If $f$ is a quasimorphism then it follows almost directly from Fekete's Lemma that
the above limit exists. The next result follows from a generalisation of Fekete's Lemma
due to de Bruijn-Erd\"os \cite{MR0047162}.

\begin{proposition}\label{P:fekete}
Let $\varphi\colon \B R_+\to \B R_+$ be an increasing function such that 
$\int_{1}^{\infty}\frac{\varphi(t)}{t^2}dt<\infty$.
Let $f\colon G\to \B R$ be such that
$$
|f(g)-f(gh)+f(h)|\leq \varphi(\|g\|+\|h\|),
$$
where $\|\ \|$ is a norm on $G$.
Then the limit \eqref{Eq:homo} exists and $f$ admits a homogenisation.
\end{proposition}

\begin{example}\label{E:walk}
A standard walk is a function $w\colon \B Z\to \B Z$ such that $w(0)=0$ and
$|w(n)-w(n+1)|=~1$. It is a partial quasimorphism relative to the
absolute value. It is easy to construct a walk such that the sequence
$\frac{w(n)}{n}$ has infinitely many convergent subsequences with
pairwise distinct limits.
\hfill $\diamondsuit$
\end{example}

\begin{example}\label{E:norm-qm}
Perhaps the most tautological example of a Lipschitz function is the norm,
$\|g\| = d(g,1)$. Hence, it follows from Theorem \ref{T:main} that the norm
is a partial quasimorphism and its homogenisation, also known as the
{\it translation length} is a homogeneous partial quasimorphism.
\hfill $\diamondsuit$
\end{example}

\begin{proposition}\label{P:homo-pqm}
Let $f\colon G\to \B R$ be a Lipschitz function with respect to a bi-invariant metric.
Let $\omega\colon 2^{\B N}\to \{0,1\}$
be a non-principal ultrafilter. Let $f_{\omega}\colon G\to \B R$ be defined by
$$
f_{\omega}(g) = \lim_{\omega}\frac{f(g^n)}{n}.
$$
If $\omega$ is a linear ultrafilter (see page \pageref{p:ultrafilter} for
definition) then $f_{\omega}$ is a homogeneous partial quasimorphism relative
to the associated norm.
If $f$ is antisymmetric then so is $f_{\omega}$.
\end{proposition}

\begin{remark}
The homogenisation $\widehat{f}$ of a quasimorphism $f$ has a very useful property
that $\sup_{g\in G}|\widehat{f}(g)-f(g)|\leq D$. There is no such control in
the general case of partial quasimorphisms or even for functions from
Proposition \ref{P:fekete}.
\end{remark}

\subsubsection*{Relation to asymptotic cones}

Relations between partial quasimorphisms on verbal subgroups and their
asymptotic cones have been studied first by Calegari and Zhuang
\cite{MR2866929}. The following observations are similar to their results.

Let $\OP{Cone}_{\omega}(G)$ be the asymptotic cone of $G$ equipped with a
bi-invariant metric $d$ with respect to a non-principal ultrafilter $\omega$.
It is a complete metric group (see page \pageref{par:acones} for more details).
Let $\eta\colon G\to \OP{Cone}_{\omega}(G)$ be defined by
$\eta(g)=[g^n]$.
Let $f\colon G\to \B R$ be a function for which there exists $C>0$ such that
$|f(g)|\leq C\|g\|$ for all $g\in G$.
Let $F_{\omega}\colon \OP{Cone}_{\omega}(G)\to \B R$ be defined by
$$
F_{\omega}[g_n]=\lim_{\omega}\frac{f(g_n)}{n}.
$$

\begin{proposition}\label{P:Fomega}
Let $f\colon G\to \B R$ be as above.
\begin{enumerate}
\item 
If $f$ is a (homogeneous) partial quasimorphism then 
$F_{\omega}$ is a (homogeneous) partial quasimorphism.
\item
If $|f(g)-f(gh)+f(h)| \leq \varphi(\|g\|+\|h\|)$, where $\frac{\varphi(n)}{n}\to 0$,
then $F_{\omega}$ is a homomorphism.
\item
$F_{\omega}\circ \eta = f_{\omega}$.
\end{enumerate}
\end{proposition}

\begin{proposition}\label{P:F-f}
Let $G$ be equipped with a bi-invariant word metric.
If $F\colon \OP{Cone}_{\omega}(G)\to~\B R$ is a Lipschitz function such that $F(1)=0$ 
then $F\circ \eta\colon G\to \B R$ is a partial quasimorphism.
\end{proposition}
\begin{remark}
If $F\colon \OP{Cone}_{\omega}(G)\to \B R$ is a Lipschitz homomorphism then $F\circ \eta$ is
also a partial quasimorphism. That is, $F$ being a homomorphism does not seem to imply a stronger
statement on $F\circ \eta$. The situation is a bit different for the commutator length.
\end{remark}

\begin{example}\label{P:cl}
If $F\colon \OP{Cone}_{\omega}([G,G],\OP{cl})\to \B R$ is a Lipschitz homomorphism then
$F\circ \eta$ is a quasimorphism. This is a special case of a result by
Calegari-Zhuang \cite[Section 3]{MR2866929}. The proof is similar to the proof of
Proposition \ref{P:F-f}.
\hfill $\diamondsuit$
\end{example}

\paragraph{Acknowledgements.}
This work was funded by Leverhulme Trust Research Project Grant RPG-2017-159.
The author was partially supported by the Polish NCN grant 2017/27/B/ST1/01467.
We thank Danny Calegari for comments on an earlier version of this paper and
the anonymous referee for careful reading and comments.

\section{Definitions and supporting results}

\paragraph{Norms.}
A function $\|\ \|\colon G\to \B R$ is called a {\em norm} if it satisfies
the following conditions for all $g,h\in G$:
\begin{enumerate}
\item $\|g\|\geq 0$ and $\|g\|=0$ if and only if $g=1_G$;
\item $\|g^{-1}\| = \|g\|$;
\item $\|gh\| \leq \|g\| + \|h\|$;
\end{enumerate}
a norm $\|\ \|$ is called {\em conjugation-invariant} if moreover
\begin{enumerate}
\setcounter{enumi}{3}
\item $\|h^{-1}gh\| = \|g\|$.
\end{enumerate}
The associated metric defined by $d(g,h)=\|gh^{-1}\|$ is right-invariant.
If $\|\ \|$ is conjugation-invariant then $d$ is bi-invariant, that is, both left-
and right-invariant. 

\begin{example}\label{E:word-norm}
If $S=S^{-1}\subseteq G$ is a symmetric generating set then
$$
\|g\|_S = \min\{n\in \B N\ |\ g=s_1\dots s_n,\ s_i\in S\}
$$
is a norm called the {\em word norm} associated with the generating set $S$.
If $S$ is normal, that is, $S=g^{-1}Sg$ for all $g\in G$ then the associated
word norm is conjugation-invariant. We will usually omit the subscript when it
does not lead to confusion.
\hfill $\diamondsuit$
\end{example}

\begin{example}\label{E:cl}
If $S$ is the set of all commutators $[g,h]\in [G,G]$, where $g,h\in G$ then
the corresponding word norm on the commutator subgroup is called {\em commutator length}
and it is relatively well understood \cite{MR2527432}. It is invariant under conjugations
by elements of~$G$.
\hfill $\diamondsuit$
\end{example}

\begin{example}\label{E:max-word}
If $S$ is a union of finitely many conjugacy classes of $G$ then the corresponding word
norm is {\it maximal} in the sense that the identity homomorphism is Lipschitz from
$\|\ \|_S$ to any other bi-invariant norm. We will refer to this word norm as the
{\em maximal word norm} on $G$.
\hfill $\diamondsuit$
\end{example}

\paragraph{Lipschitz functions.}\label{p:lipschitz}
A function $f\colon G\to \B R$ is {\em Lipschitz} with constant $C>0$ if
$|f(g)-f(h)|\leq Cd(g,h)$ holds for all $g,h\in G$. It is Lipschitz
{\it with respect to the norm} if $|f(g)|\leq C\|g\|$ holds for all
$g\in G$.

\paragraph{Ultrafilters.}\label{p:ultrafilter}
An {\em ultrafilter} $\omega$ on the set $\B N$ of natural numbers is
a maximal (with respect to inclusion) 
filter on the power set $2^{\B N}$ (with the partial order given by inclusion of sets).
Equivalently it is a finitely additive measure $\omega \colon 2^{\B N}\to \{0,1\}$.
The equivalence is given by saying that sets of full measure belong to the filter.
An ultrafilter is called non-principal if every finite set is of measure zero.
If $a\colon \B N\to X$ is a bounded sequence in a metric space 
$(X,d)$ then the {\em ultralimit} with respect to an ultrafilter $\omega$ is
defined by following condition:
$\lim_{\omega} a_n = a$
if and only if for every $\varepsilon >0$
$$
\omega\{n\in \B N\ |\ d(a_n,a)<\varepsilon\} = 1.
$$

An ultrafilter on $\B N$ is called {\em linear} if it contains all sets of the
form $k\B N$, where $0<k\in \B N$. Such ultrafilter exists. Indeed, let $\C F$
be a collection containing all sets of the form $0<k\B N$, where $k\in \B N$,
and all their supersets. If $A,B\in \C F$ then $k_A\B N\subseteq A$ and 
$k_B\B N\subseteq B$, for some $k_A,k_B\in \B N\setminus\{0\}$. Then
$$
\OP{lcm}(k_A,k_B)\B N=k_A\B N\cap k_B\B N \subseteq A\cap B,
$$
which shows that $A\cap B\in \C F$ and hence it is a filter since it is
nonempty and upward closed by definition. It is thus contained in some 
non-principal ultrafilter.

\paragraph{Asymptotic cones.}\label{par:acones}
Let $\omega$ be a non-principal ultrafilter.
Let $G$ be equipped with a bi-invariant metric $d$ and let $\|\ \|$ denote the
associated norm.
Let $\prod_0 G = \{(g_n)\in G^{\B N}\ |\ \|g_n\| = O(n)\}$ be the set of
sequences of elements of $G$ such that their norms grow at most linearly.
It is a group with pointwise multiplication and
$$
\|(g_n)\|_{\omega} = \lim_{\omega}\frac{\|g_n\|}{n}
$$
defines a degenerate norm, where degenerate means that some elements
can have norm equal to zero. These elements form a normal subgroup
and the corresponding quotient group $\OP{Cone}_{\omega}(G,d)$ is
called the {\em asymptotic cone} of $(G,d)$ \cite{MR2866929}.
It is a complete metric group with the metric
$d_{\omega}$ associated to the norm $\|\ \|_{\omega}$ defined above.
See Dru\c{t}u-Kapovich \cite{MR3753580} for a general and systematic approach
to asymptotic cones of metric spaces. See also the thesis of Jakob Schneider
\cite{schneider} for a metric ultraproduct approach.

\begin{example}\label{E:acone}
In general, very little is known about the topological or algebraic structure
of asymptotic cones of groups with bi-invariant metrics.  Here is a sample of
relatively easy facts.

\begin{enumerate}
\item 
The asymptotic cone of a free abelian group $\B Z^n$ equipped with its standard word
metric is isometric to $\B R^n$ with the $L^1$-metric.

\item 
Let $G$ be a finitely generated nilpotent group equipped with the maximal
bi-invariant word metric. The abelianisation $G\to G/[G,G]$ is a quasi-isometry
and hence the asymptotic cone of $G$ is isometrically isomorphic to $\B R^n$,
where $n$ is the rank of the abelianisation.

\item
The asymptotic cone of the infinite symmetric group $\B S_{\infty}$ equipped
with a bi-invariant word metric associated with any finite normally generating set is
a simple contractible metric group \cite[Theorem 5.1]{2203.10889}.

\item 

The asymptotic cone $\OP{Cone}(\B F_2)$ of the free group $\B F_2=\langle
a,b\rangle$ on two generators with respect to the maximal bi-invariant word
metric is non-separable. Indeed, there is a quasi-isometric embedding of a
regular tree $T\to \OP{Cay}(\B F_2)$ and hence the asymptotic cone contains an
isometrically embedded $\B R$-tree.  Since the section $\B Z^2\to \B F_2$ of
the abelianisation given by $(m,n) \mapsto a^mb^n$ is an isometric embedding
the cone $\OP{Cone}_{\omega}(\B F_2)$ contains many flats, i.e., isometrically
embedded copies of $\B R^2$.

\item 
The asymptotic cone $\OP{Cone}_{\omega}([G,G],\OP{cl})$ of $[G,G]$ equipped
with the commutator length is abelian.
\end{enumerate}
\hfill $\diamondsuit$
\end{example}

\paragraph{Word norms on length groups.}
\begin{lemma}\label{L:length-word}
Let $G$ be a metric group with a bi-invariant metric $d$. Assume that
$(G,d)$ is a length space.
Let $S=B(1)\subseteq G$ be the ball of radius $1$ centred at the identity
and let $d_S$ denote the corresponding word metric. Then the identity is
a quasi-isometry between $d$ and $d_S$. More precisely, we have
$$
\|g\|\leq \|g\|_S \leq \|g\|+1.
$$
\end{lemma}

\begin{proof}
Let $\|g\|=d(g,1)$ and let $\|g\|_S=d_S(g,1)$ be the corresponding norms.
Let $\C L$ denote the length of rectifiable paths.
Since the identity is a homomorphism, it is enough to verify the statement
for the norms. 

Let $\varepsilon>0$. Let $g\in G$ and let $\{g_t\}$ be a path
from the identity to $g$ such that $\C L\{g_t\}\leq \|g\|+\varepsilon$. 
By subdividing the path into segments of length
$1$ and the last segment of possibly smaller length we see that
$\|g\|_S \leq \|g\| + 1$. 

If $\|g\|_S=n$ then $g=s_1\dots s_n$, where $\|s_i\|\leq 1$. Thus there is a path
from the identity to $g$ of length at most $n+\varepsilon$ which shows that
$\|g\|\leq n = \|g\|_S$. Thus the identity is a quasi-isometry.
\end{proof}

\begin{corollary}\label{C:acone}
The metrics $d_{\omega}$ and the word metric $d_{\omega S}$ on the 
asymptotic cone $\OP{Cone}_{\omega}(G)$ are quasi-isometric.
\qed
\end{corollary}

\paragraph{A useful identity.}
The following identity will be used several times in subsequent proofs.
Its straightforward proof is left to the reader.
\begin{lemma}\label{L:c-trick}
Let $G$ be equipped with a conjugation-invariant norm $\|\ \|$.
Then for all $g,h\in G$ and $n\in \B N$
\begin{equation}
g^nh^n = (gh)^n c_1 \dots c_{n-1},
\label{Eq:c-trick}
\end{equation}
where all $c_i$'s are conjugates of $[g,x_i]$ for some $x_i\in G$ or
all $c_i$'s are conjugates of $[h,x_i]$ for some $x_i\in G$.  In particular,
$$
\|c_1\dots c_n \| \leq 2(n-1)\min\{\|g\|,\|h\|\}.
$$
\qed
\end{lemma}

\section{Proofs}

\subsection{Proof of Theorem \ref{T:main}}
Suppose that $f\colon G\to \B R$ is a partial quasimorphism bounded on the generating set.
Let $\sigma = \sup_{s\in S}f(s)$. Let $\|h\|=n$ and let $h=s_1\dots s_n$ for some $s_i \in S$.
First observe that $f$ is Lipschitz with respect to the norm.
\begin{align*}
|f(h)| &= |f(s_1\dots s_n)|\\
&\leq \sum_{i=1}^n |f(s_i)| + (n-1) D\\
&\leq (\sigma + D) n = (\sigma + D)\|h\|.
\end{align*}
Now we use the above to show that $f$ is Lipschitz.
\begin{align*}
|f(g) - f(gh)| &\leq |f(g) - f(gh) + f(h)| + |f(h)|\\
&\leq D\|h\| + (\sigma + D)\|h\|\\ 
&= (\sigma + 2D)\|h\| = (\sigma + 2D)d(g,gh)
\end{align*}

Conversely, assume that $f$ is Lipschitz with constant $C>0$. Let $s\in S$ be a generator. 
\begin{align*}
|f(s)| - |f(1)| &\leq |f(s)-f(1)| \leq Cd(s,1) = C\\
|f(s)| &\leq C + |f(1)|\\
\end{align*}
which shows that $f$ is bounded on the generating set.
\begin{align*}
|f(g) - f(gh) + f(h)| 
&\leq |f(g) - f(gh)| + |f(h)-f(1)| + |f(1)|\\
&\leq C\|h\| + C\|h\| + |f(1)|\\
&\leq (2C + |f(1)|) \|h\|.
\end{align*}
In the last inequality we use the fact that $d$ is a word metric, hence, $\|h\|\geq 1$ if $h\neq 1_G$.
A similar computation and the assumption that $d$ is bi-invariant 
shows the bound by $(2C + |f(1)|)\|g\|$. The bi-invariance implies that
$d(g,gh)=\|h\|$ and $d(gh,h)=\|g\|$. This proves the first statement of the theorem.

Let $f\colon G\to \B R$ be an antisymmetric homogeneous quasimorphism such that $|f(g)|>0$ 
and $f$ is bounded (by $\sigma$) on the generating set.
Then
$n|f(g)| = |f(g^n)|\leq (\sigma + D)\|g^n\|$, which shows that $g$ is undistorted.
Notice that the antisymmetry is not used here. 

Conversely, let $g$ be undistorted, that is, $\|g^n\|\geq c|n|$ for some $c>0$. 
Let $f'\colon \langle g\rangle \to \B R$ be defined by $f'(g^n) = cn$ for all $n\in \B Z$.
Then $f'$ is antisymmetric homogeneous Lipschitz on $\langle g\rangle$, the subgroup
generated by $g$, equipped with the induced metric. Indeed,
\begin{align*}
|f'(g^n) - f'(g^m)| &= |cn-cm| \leq \|g^{n-m}\| = d(g^n,g^m),\\
\end{align*}
which shows that $f'$ is Lipschitz with constant $1$. Antisymmetry and homogeneity are clear.
In the following part of the proof we construct an extension $f$ of $f'$ that is also Lipschitz.
Let $f\colon G\to \B R$ be defined by
$$
f(h) = \inf \{f'(g^n) + d(h,g^n)\ |\ n\in \B Z\}
$$
First we show that $f$ is well defined.
\begin{align*}
f'(g^n) + d(h,g^n) 
&= f'(g) - f'(g) + f'(g^n) + d(h,g^n)\\
&\geq f'(g) - d(g,g^n) + d(h,g^n) \\
&\geq f'(g) - d(g,h) > -\infty.\\
\end{align*}
Moreover $f(g^n) = f'(g^n)$, so $f$ is an extension of $f'$.

Let $\varepsilon >0$ and let $h,h'\in G$ be any two elements. Without loss of
generality, we can assume that $f(h')-f(h)\geq 0$.  Let $g_h\in \langle
g\rangle$ be an element such that
$$
f(h) \geq f(g_h) + d(h,g_h)-\epsilon.
$$
Then by definition of $f$ we have also that $f(h')\leq f(g_h)+d(h',g_h)$. 
Consequently,
\begin{align*}
0\leq f(h') - f(h) &\leq f(g_h) + d(h',g_h) - f(g_h) - d(h,g_h) + \epsilon \\
&\leq d(h',g_h) + d(g_h,h) + \epsilon \\
&\leq d(h',h) + \epsilon,
\end{align*}
which shows that $f$ is Lipschitz (with constant $1$) and bounded
on the generating set.

Let $\overline{f}(g) = \frac{1}{2}\left( f(g)-f(g^{-1}) \right)$ be the {\em anti-symmetrisation}
of $f$. Homogenising $\overline{f}$ with respect
to a linear ultrafilter $\omega$ as in Proposition \ref{P:homo-pqm} does not change $f'$
and hence~$\overline{f}_{\omega}$ is the required antisymmetric homogeneous partial quasimorphism.
Boundedness on the generating set is clearly preserved by anti-simetrisation and homogenisation.
\qed

\subsection{Proof of Proposition \ref{P:fekete}}

Theorem 23 in \cite{MR0047162} states that if $\varphi \colon \B R_+\to \B R_+$
is an increasing function such that 
\begin{equation}
\int_1^{\infty}\frac{\varphi(t)}{t^2}dt <\infty
\label{Eq:fekete}
\end{equation}
and $a\colon \B N\to \B R$ is a sequence such that
$$
a(m+n) \leq a(m) + a(n) + \varphi(m+n) 
$$
then $-\infty \leq \lim_{n\to \infty} \frac{a(n)}{n}< \infty$.
Let $f\colon G\to \B R$ be a function satisfying the hypothesis of Proposition \ref{P:fekete}
and let $g\in G$.
It follows that
\begin{align*}
|f(g^m) - f(g^{m+n}) + f(g^n)| &\leq \varphi \left( \|g^m\|+\|g^n\| \right)\\
f(g^{m+n}) &\leq f(g^m) + f(g^n) + \varphi \left( \|g^m\|+\|g^n\| \right)\\
&\leq f(g^m) + f(g^n) + \varphi \left( (m+n)\|g\| \right).\\
\end{align*}
Since the function $n\mapsto \varphi(n\|g\|)$ also satisfies the integral
condition \eqref{Eq:fekete} the statement of the de Brujin-Erd\"os applies.
If we use only the above subadditivity then can only conclude that
the limit $\lim_{n\to \infty}\frac{f(g^n)}{n}$ either exists or it is $-\infty$.
However, we also have that
$$
f(g^m) + f(g^n) \leq f(g^{m+n}) + \varphi \left( \|g^m\|+\|g^n\| \right)
$$
and the modification of de Bruijn-Erd\"os proof with the use of Fekete's Lemma for
superadditive sequences shows that the limit is bigger than~$-\infty$.\qed

\subsection{Proof of Proposition \ref{P:homo-pqm}}

Since $f$ is Lipschitz and the associated norm satisfies the triangle
inequality, we have that $|f(g^n)| \leq C\|g^n\| \leq C\|g\|n$. Consequently,
the sequence $\frac{f(g^n)}{n}$ is bounded and its $\omega$-limit exists.  
The following computation shows that $f_{\omega}$ is a partial quasimorphism.
\begin{align*}
|f_{\omega}(g) - f_{\omega}(gh) + f_{\omega}(h)|
&= \lim_{\omega} \frac{|f(g^n)-f\left( (gh)^n\right)+f(h^n)|}{n}\\
&= \lim_{\omega} \frac{|f(g^n)-f(g^nh^nc_1\dots c_{n-1})+f(h^n)|}{n}\\
&\leq \lim_{\omega} \frac{|f(g^n)-f(g^nh^n)+f(h^n)|+|f(c_1\dots c_{n-1})|}{n}\\
&\leq \lim_{\omega} \frac{D\min\{\|g^n\|,\|h^n\|\}+C\|c_1\dots c_{n-1}\|}{n}\\
&\leq \lim_{\omega} \frac{D\min\{n\|g\|,n\|h\|\}+2C\min\{(n-1)\|g\|,(n-1)\|h\|\}}{n}\\
&\leq (2C+D)\min\{\|g\|,\|h\|\}.\\
\end{align*}
If $\omega$ is linear and $k\in \B N$ then we have that
$$
f_{\omega}(g^k) 
= \lim_{\omega}\frac{f(g^{kn})}{n}
= k\lim_{\omega}\frac{f(g^{kn})}{kn}
= k\lim_{\omega}\frac{f(g^{n})}{n}
$$
since $\omega(k\B N)=1$. This shows that $f_{\omega}$ is homogeneous.

\subsection{Proof of Proposition \ref{P:Fomega}}

If $f\colon G\to \B R$ is a partial quasimorphism then the following computation
shows that $F_{\omega}$ is also a partial quasimorphism. 
\begin{align*}
|F_{\omega}[g_n] - F_{\omega}[g_nh_n] + F_{\omega}[h_n]|
&\leq \lim_{\omega} \frac{|f(g_n)-f(g_nh_n)+f(h_n)|}{n}\\
&\leq \lim_{\omega} \frac{D\min\{\|g_n\|,\|h_n\|\}}{n}\\
&\leq D\min\{\|[g_n]\|_{\omega},\|[h_n]\|_{\omega}\}.
\end{align*}
It is straightforward to see that if $f$ homogeneous then so is $F_{\omega}$. This proves the
first item. The second item also follows from the above computation with the assumed estimate.
The third item is immediate.
\qed

\subsection{Proof of Proposition \ref{P:F-f}}
Let $F\colon \OP{Cone}_{\omega}(G)\to \B R$ be a Lipschitz function with constant $C>0$
and such that $F(1)=0$.
If follows from Lemma \ref{L:length-word} that it is also Lipschitz with respect
to the word metric $d_{\omega S}$ with the same constant $C$. Furthermore,
Theorem \ref{T:main} implies that $F$ is a partial quasimorphism relative
to $d_{\omega S}$ with constant $2C + F(1)=2C$.

Assume that neither $g$ or $h$ is equal to the identity. Otherwise the computation
below is trivial. Moreover, since the metric on $G$ is a word metric, we have that
$\|g\|+1\leq 2\|g\|$ for $g\neq 1_G$, which is used below. Since
$C=\sup_{s\in S}F(s)$, we have that $|F(g)|\leq 3C\|g\|_{\omega S}$, (see the beginning of 
the proof of Theorem \ref{T:main}), which we also use below.
The following computation finishes the proof.
\begin{align*}
|F(\eta(g)) - F(\eta(gh)) +& F(\eta(h))|\\
&=|F[g^n] - F[(gh)^n] + F[h^n]|\\
&=|F[g^n] - F[g^nh^nc_1\dots c_{n-1}] + F[h^n]|\\
&\leq |F[g^n] - F[g^nh^n] + F[h^n]| + |F[c_1\dots c_{n-1}]| + 2C\|c_1\dots c_{n-1}\|_{\omega S}\\
&\leq 2C\min\{\|[g^n]\|_{\omega S},\|[h^n]\|_{\omega S}\} + 5C\|c_1\dots c_{n-1}\|_{\omega S}\\ 
&\leq 2C\min\{\|[g^n]\|_{\omega}+1,\|[h^n]\|_{\omega}+1\} + 5C(\|c_1\dots c_{n-1}\|_{\omega}+1)\\ 
&\leq 2C\min\{\|g\|+1,\|h\|+1\} + 5C\left(\lim_{\omega}\frac{\|c_1\dots c_{n-1}\|}{n}+1\right)\\ 
&\leq 4C\min\{\|g\|,\|h\|\} + 5C\left(\lim_{\omega}\frac{2n\min\{\|g\|,\|h\|\}}{n}+1\right)\\ 
&\leq 4C\min\{\|g\|,\|h\|\} + 20C\min\{\|g\|,\|h\|\}\\ 
&\leq 24C\min\{\|g\|,\|h\|\}.\\
\end{align*}
\qed

\bibliography{/home/kedra/sync/bibliography}
\bibliographystyle{plain}

\bigskip
\bigskip
\noindent
Jarek K\k{e}dra\\ 
University of Aberdeen and University of Szczecin \\
Email: {\tt kedra@abdn.ac.uk}
\end{document}